\documentclass[a4paper,12pt]{article}

\makeatletter
\def\normalsize{\@setfontsize\normalsize\@xipt{11}}
\makeatother
\usepackage{amssymb}
\usepackage{amsthm}
\usepackage{amsmath}
\usepackage{amsfonts}
\usepackage{setspace}  
\usepackage{anysize}
\usepackage{vmargin}
\usepackage{caption}

\setpapersize{A4} 
\setmarginsrb{4cm}{3cm}{2.5cm}{3cm}{13pt}{19pt}{13pt}{27pt}

\setpapersize{A4} 
\setmarginsrb{2.5cm}{2.5cm}{2.5cm}{2.5cm}{0pt}{0.7cm}{0pt}{0.7cm}

\theoremstyle{plain}
\newtheorem{theorem}{Theorem}
\newtheorem{lemma}{Lemma}

\theoremstyle{definition}

\theoremstyle{remark}

\begin{document}

\begin{center}
\LARGE
Eigenvalue Asymptotics for the Schr\"{o}dinger Operator with a Matrix Potential in a Single Resonance Domain
\end{center}
\vspace{1cm}

\normalsize
\centerline{Sedef KARAK{\l}L{\l}\c{C}}\centerline{Department of
Mathematics, Faculty of Science, Dokuz Eyl\"{u}l
University,} \centerline{T{\i}naztepe Camp., Buca, 35160, Izmir,
Turkey} \centerline{sedef.erim@deu.edu.tr}
\vspace{0.5cm}

\centerline{Setenay AKDUMAN} \centerline{Department of Mathematics,
Faculty of  Science, Dokuz Eyl\"{u}l University,}
\centerline{T{\i}naztepe Camp., Buca, 35160, Izmir, Turkey}
\centerline{setenayakduman@gmail.com}

\begin{abstract}

We consider a Schr\"{o}dinger Operator with a matrix potential defined  in $L_2^m(F)$ by the differential expression
    \begin{equation*}
    L(\phi(x))=(-\Delta+V(x))\phi(x)
    \end{equation*}
     and the Neumann boundary condition, where $F$ is the $d$ dimensional rectangle and $V$ is a martix potential, $m\geqslant 2, d\geqslant 2$. We obtain the asymptotic formulas of arbitrary order for the single resonance eigenvalues of the Schr\"{o}dinger operator in $L_2^m(F)$. \\



\end{abstract}
\newpage
\begin{center}
\large
Introduction
\end{center}
\normalsize
\par We consider the Schr\"{o}dinger Operator with a matrix potential $V(x)$ defined  by the differential expression
    \begin{equation}
    L(\phi(x))=(-\Delta+V(x))\phi(x)
    \end{equation}
     and the Neumann boundary condition
     \begin{equation}
     \frac{\partial \phi }{\partial n} |_{\partial F}=0,
     \end{equation}
     in $L_2^m(F)$ where $F$ is the $d$ dimensional rectangle $F=[0,a_1]\times[0,a_2]\times\ldots \times [0,a_d],$ $\partial F$ is the  boundary of $F$, $m\geqslant 2, d\geqslant 2$,  $\frac{\partial}{\partial n}$ denotes differentiation along the outward normal of the boundary $\partial F$, $\boldsymbol{\Delta}$ is a diagonal $m \times m$ matrix whose diagonal elements are the scalar Laplace operators $\Delta=\frac{\partial ^2}{\partial {x_1}^2}+\frac{\partial ^2}{\partial {x_2}^2}+\ldots+\frac{\partial ^2}{\partial {x_d}^2}$ , $x=(x_1,x_2,\ldots,x_d)\in \boldsymbol{R}^d$, $V$ is a real valued symmetric matrix $V(x)=(v_{ij}(x)), i, j=1,2,\ldots,m, v_{ij}(x)\in L_2(F)$, that is, $V^T(x)=V(x).$
     \par We denote the operator defined by (1)-(2) by $L(V)$, and the eigenvalues and corresponding eigenfunctions of $L(V)$ by $\Lambda_{N}$ and $\Psi_{N}$, respectively.
     \par The eigenvalues of the operator $L(0)$ which is defined
     by (1) when $V(x)=0$ and the boundary condition (2)
     are $\mid \gamma\mid^{2}$ and the corresponding eigenspaces are
     \begin{equation*}
     E_{\gamma}=span\{\Phi_{\gamma,1}(x),\Phi_{\gamma,2}(x),\ldots,\Phi_{\gamma,m}(x)\},
     \end{equation*}
     where $\gamma \in
     \frac{\Gamma^{+0}}{2}=\{(\frac{n_{1}\pi}{a_{1}},\frac{n_{2}\pi}{a_{2}}\cdots,\frac{n_{d}\pi}{a_{d}}):\quad
     n_{k}\in Z^{+}\bigcup\{0\},\quad k=1,2,\ldots,d\}$,\\
     \\$\Phi_{\gamma,j}(x)=(0,\ldots,0,u_{\gamma}(x),0,\ldots,0)$,\quad
     $j=1,2,\ldots,m$,\quad$u_{\gamma}(x)=cos
     \frac{n_{1}\pi}{a_{1}}x_{1}cos \frac{n_{2}\pi}{a_{2}}x_{2}\cdots cos
     \frac{n_{d}\pi}{a_{d}}x_{d}$, $u_{0}(x)=1$ when
     $\gamma=(0,0,\ldots,0)$. We note that the non-zero component
     $u_{\gamma}(x)$ of $\Phi_{\gamma,j}(x)$ stands in the $j$th
     component.
\par It can be easily calculated that the norm of $u_{\gamma}(x)$, $\gamma=(\gamma^{1},\gamma^{2},\ldots,\gamma^{d})\in\frac{\Gamma^{+0}}{2}$
in $L_{2}(F)$ is $\sqrt{\frac{\mu(F)}{\mid A_{\gamma}\mid}}$, where $\mu(F)$ is the measure of the $d$-dimensional parallelepiped $F$, $\mid A_{\gamma}\mid$ is the number of vectors in $A_{\gamma}=\{\alpha=(\alpha_{1},\alpha_{2},\ldots,\alpha_{d})\in\frac{\Gamma}{2}:
\quad \mid\alpha_{k}\mid=\mid\gamma^{k}\mid,\quad k=1,2,\ldots,d\}$\;\; and \\ $\frac{\Gamma}{2}=\{(\frac{n_{1}\pi}{a_{1}},\frac{n_{2}\pi}{a_{2}}\cdots,\frac{n_{d}\pi}{a_{d}}):\quad
n_{k}\in Z,\quad k=1,2,\ldots,d\}$.
\par Since $\{u_{\gamma}(x)\}_{\gamma\in\frac{\Gamma^{+0}}{2}}$ is a complete system in $L_{2}(F)$, for any $q(x)$
in $L_{2}(F)$ we have
\begin{equation}\label{q1}
q(x)=\sum_{\gamma\in\frac{\Gamma^{+0}}{2}}\frac{\mid
A_{\gamma}\mid}{\mu(F)}(q,u_{\gamma})u_{\gamma}(x),
\end{equation}
where $(\cdot,\cdot)$ is the inner product in $L_{2}(F)$.
\par In our study, it is convenient to use the equivalent decomposition (see \cite{Karalilic1})
\begin{equation}\label{q2}
q(x)=\sum_{\gamma\in\frac{\Gamma}{2}}q_{\gamma}u_{\gamma}(x),
\end{equation}
where $q_{\gamma}=\frac{1}{\mu(F)}(q(x),u_{\gamma}(x))$ for the sake
of simplicity. That is, the decomposition \eqref{q1} and \eqref{q2} are equivalent for any $d\geq 1.$
\par Each matrix element $v_{ij}(x)\in L_2(F)$ of the matrix $V(x)$ can be written in its Fourier series expansion
\begin{equation}\label{sumv}
    v_{ij}(x)=\sum_{\gamma\in
\frac{\Gamma}{2}}v_{ij\gamma}u_{\gamma}(x)
\end{equation}
for $i,j=1,2,\ldots,m$ where
$v_{ij\gamma}=\frac{(v_{ij},u_{\gamma})}{\mu(F)}$.
\par We assume that the Fourier coefficients $v_{ij\gamma}$ of $v_{ij}(x)$ satisfy
\begin{equation}\label{coeff}
\sum_{\gamma\in\frac{\Gamma}{2}}\mid v_{ij\gamma}\mid^{2}(1+\mid
\gamma\mid^{2l})<\infty,
\end{equation}
for each $i,j=1,2,\ldots,m$,\;\;  $l>\frac{(d+20)(d-1)}{2}+d+3$
which implies
\begin{equation}\label{decomp}
v_{ij}(x)=\sum_{\gamma\in\Gamma^{+0}(\rho^{\alpha})}v_{ij\gamma}u_{\gamma}(x)+
O(\rho^{-p\alpha}),
\end{equation}
where $\Gamma^{+0}(\rho^{\alpha})=\{\gamma\in\frac{\Gamma}{2}:0\leq\mid
\gamma\mid<\rho^{\alpha}\}$, $p=l-d$, $\alpha<\frac{1}{d+20}$,
$\rho$ is a large parameter and $O(\rho^{-p\alpha})$ is a function
in $L_{2}(F)$ with norm of order $\rho^{-p\alpha}$. Furthermore, by \eqref{coeff}, we have
\begin{equation}\label{decompsum}
    M_{ij}\equiv\sum_{\gamma\in\frac{\Gamma}{2}}\mid v_{ij\gamma}\mid<\infty,
\end{equation}
for all $i,j=1,2,\ldots,m$.
\par Notice that, if a function $q(x)$ is sufficiently smooth $\left(q(x)\in W_2^l(F)\right)$ and the support of $gradq(x)=\left(\frac{\partial q}{\partial x_1},\frac{\partial q}{\partial x_2},\ldots,\frac{\partial q}{\partial x_d}\right)$ is contained in the interior of the domain $F$, then $q(x)$ satisfies condition \eqref{coeff} (See \cite{Hald}). There is also another class of functions $q(x)$, such that $q(x)\in W_2^l(F)$,
\begin{align*}
q(x)=\sum_{\gamma^{'}\in \Gamma }q_{\gamma^{'}}u_{\gamma^{'}}(x),
\end{align*}
which is periodic with respect to a lattice $\Omega=\left\{(m_1a_1,m_2a_2,\ldots,m_da_d):m_k\in \boldsymbol{Z},k=1,2,\ldots,d \right\}$
and thus it also satisfies condition \eqref{coeff}.
\par One of the essential problems related to this operator $L(V)$ is how the eigenvalues $\lvert \gamma \lvert^2$ of the unperturbed operator $L(0)$ is affected under perturbation. We study this problem by using energy as a large parameter, in other words when $\lvert \gamma \lvert \sim \rho$, that is, there exist positive constants $c_1,$ $c_2$ such that $c_1 \rho < \lvert \gamma \lvert  <c_2\rho$, \; $c_1,$ $c_2$ do not depend on $\rho$ and $\rho$ is a big parameter. In the sequel, we denote by $c_i$, $i=1,2,\ldots,$ the positive constants which does not depend on $\rho.$
\par For the scalar case, $m=1$, a method in which for the firts time the eigenvalues of the unperturbed operator $L(0)$ were divided into two groups: non-resonance ones and resonance ones
 was first introduced by O. Veliev in \cite{Veliev1}
and more recently in \cite{Veliev4}, \cite{Veliev5} to obtain
various asymptotic formulas for the eigenvalues of the periodic
Schr\"{o}dinger operator with quasiperiodic boundary conditions
corresponding to each group. By some other methods, \\asymptotic
formulas for quasiperiodic boundary conditions in two and three
dimensional cases are obtained in \cite{Feldman1},
\cite{Friedlanger}, \cite{Karpeshina1}, \cite{Karpeshina2} and
\cite{Hald}. When this operator is considered with Dirichlet
boundary condition in two dimensional rectangle, the asymptotic
formulas for the eigenvalues are obtained in \cite{Hald}. The
asymptotic formulas for the eigenvalues of the Schr\"{o}dinger
operator with Dirichlet or Neumann boundary conditions in an
arbitrary dimension are obtained in \cite{Atilgan},
\cite{Karalilic1} and \cite{Karakilic2}. For the matrix case,
asymptotic formulas for the eigenvalues of the Schr\"{o}dinger
operator with quasiperiodic boundary conditions are obtained in
\cite{Karpeshina2}.
\par As in \cite{Veliev1}- \cite{Veliev5}, we divide $R^{d}$ into two domains:
Resonance and Non-resonance domains. \\In order to define these
domains,
let us introduce the following sets:
\par Let $\alpha<\frac{1}{d+20}$, $\alpha_{k}=3^{k}\alpha$,
$k=1,2,\ldots,d-1$ and
\begin{center}
$V_{b}(\rho^{\alpha_{1}})\equiv\left\{x\in R^{d}:\quad\left\lvert\lvert
x \lvert
^{2}- \lvert x+b \lvert^{2}\right \lvert<\rho^{\alpha_{1}}\right\}$,
\end{center}
\begin{center}
$E_{1}(\rho^{\alpha_{1}},p)\equiv\bigcup\limits_{{b}\in\Gamma(p\rho^{\alpha})}$
$V_{b}(\rho^{\alpha_{1}})$,
\end{center}
\begin{center}
$U(\rho^{\alpha_{1}},p)\equiv R^{d}\setminus
E_{1}(\rho^{\alpha_{1}},p)$,
\end{center}
\begin{center}
$E_k(\rho^{\alpha_k},p)=\bigcup\limits_{\gamma_1,\gamma_2,\ldots,\gamma_k\in \Gamma(p\rho^{\alpha})}\left(\bigcap\limits_{i=1}^kV_{\gamma_i}(\rho^{\alpha_k})\right)$,
\end{center}
where $\Gamma(p\rho^{\alpha})\equiv\left\{b\in\frac{\Gamma}{2}:0<\mid
b\mid<p\rho^{\alpha}\right\}$ and the intersection $\bigcap\limits_{i=1}^kV_{\gamma_i}(\rho^{\alpha_k}) $ in $E_k$ is taken over $\gamma_1,\gamma_2,\ldots,\gamma_k$ which are linearly independent vectors and the length of $\gamma_i$ is not greater than the length of the other vector in $\Gamma \bigcap \gamma_i\boldsymbol{R}.$ The set $U(\rho^{\alpha_{1}},p)$ is said to be a non-resonance domain, and the eigenvalue $\lvert \gamma \lvert^2$ is called a non-resonance eigenvalue if $\gamma\in U(\rho^{\alpha_{1}},p).$ The domains $V_{b}(\rho^{\alpha_{1}})$, for $b\in \Gamma(p\rho^{\alpha})$ are called resonance domains and the eigenvalue $\lvert \gamma \lvert^2$ is a resonance eigenvalue if $\gamma \in V_{b}(\rho^{\alpha_{1}})$.
\par As noted in \cite{Veliev4} and \cite{Veliev5}, the domain $V_{b}(\rho^{\alpha_{1}})\setminus E_2$, called a single resonance domain, has asymptotically full measure on $V_{b}(\rho^{\alpha_{1}}),$ that is,
\begin{align*}
\frac{\mu\left(\left(V_{b}(\rho^{\alpha_{1}})\setminus E_2\right) \bigcap B(q)\right) }{\mu\left(V_{b}(\rho^{\alpha_{1}})\bigcap B(q) \right)}\to 1,\text{ as  } \rho\to \infty,
\end{align*}
where $B(\rho)=\left\{x\in \boldsymbol{R}^d:\lvert x \lvert =\rho \right\}$, if
\begin{align}\label{alfa}
2\alpha_2-\alpha_1+(d+3)\alpha<1 \text{  and  } \alpha_2>2\alpha_1,
\end{align}
hold. Since $\alpha<\frac{1}{d+20}$, the conditions in \eqref{alfa} hold.
\par When $m\geq 2$, in \cite{coskan1}, in an arbitrary dimension, the asymptotic formulas of arbitrary order for the eigenvalue of the operator $L(V)$ which corresponds to the non-resonance eigenvalue $\lvert \gamma \lvert^2$ of $L(0)$ are obtained.
 \par In this paper, we obtain the high energy asymptotics of arbitrary order in an arbitrary dimension $(d\geq 2)$ for the eigenvalue of $L(V)$ corresponding to resonance eigenvalue  $\lvert \gamma \lvert^2$ when $\gamma$ belongs to the single resonmance domain, that is, $\gamma\in V_{\delta}(\rho^{\alpha_1})\setminus E_2,$ where $\delta$ is from $\{e_1,e_2,\ldots,e_d\}$ and $e_1=\left(\frac{\pi}{a_1},0,\ldots,0\right),\ldots,e_d=\left(0,\ldots,\frac{\pi}{a_d}\right)$.
\begin{center}
\large
Eigenvalues In a Special Single Resonance Domain
\end{center}
\normalsize
\par Now let $H_\delta=\{x\in \boldsymbol{R}:  x \cdot \delta=0\}$ be the hyperplane which is orthogonal to $\delta$. Then we define the following sets:
\begin{equation*}
\Omega_\delta=\{\omega \in \Omega: w\cdot \delta=0 \}=\Omega \cap H_\delta,
\end{equation*}
\begin{equation*}
\Gamma_\delta=\{\gamma \in \frac{\Gamma}{2}: \gamma \cdot \delta=0\}=\frac{\Gamma}{2}\cap H_\delta.
\end{equation*}
Here
$`` \cdot "$ denotes the inner product in $\boldsymbol{R}^d$. Clearly, for all $\gamma \in \frac{\Gamma}{2}$, we have the following decomposition
\begin{equation}\label{gamma}
\gamma=j\delta+\beta, \beta \in \Gamma_\delta, j \in \boldsymbol{Z}.
\end{equation}
Note that; if $\gamma=j\delta+\beta\in V_{\delta}(\rho^{\alpha_1})\backslash E_2$ then
\begin{equation}\label{single}
\lvert j \lvert < r_1 ,\;\;\; r_1=\rho^{\alpha_1}\lvert \delta\lvert^{-2}+1,\;\;\lvert\beta^k \lvert > \frac{1}{3}\rho^{\alpha_1},\;\;\; \forall k:e_k \ne \delta.
\end{equation}
We write the decomposition (3) of $v_{ij}(x)$ as
\begin{equation}\label{vij}
  v_{ij}(x)=\sum_{\gamma^{'}\in
\frac{\Gamma}{2}}v_{ij\gamma^{'}}u_{\gamma^{'}}(x)=p_{ij}(s)+\sum_{\gamma \in \frac{\Gamma}{2} \setminus \delta
 \boldsymbol{R}}v_{ij\gamma}u_{\gamma}(x)
\end{equation}
where
\begin{equation}\label{pij}
p_{ij}(s)=\sum_{n \in \boldsymbol{Z}}p_{ijn}\cos ns,\;\; p_{ijn}=v_{ij(n\delta)}, \;\;s= x\cdot\delta,\; i,j=1,2,\ldots,m.
\end{equation}
\par In order to obtain the asymptotic formulas for the single resonance eigenvalues  $\lvert \gamma \lvert^2$ $\left(\gamma \in V_{\delta}(\rho^{\alpha_1})\setminus E_2\right)$, we consider the operator $L(V)$ as the perturbation of $L(P(s))$ where $L(P(s))$ is defined by the differential expression
\begin{equation}
Lu=-\Delta u+ P(s) u
\end{equation}
and the Neumann boundary condition
\begin{equation*}
\frac{\partial u}{\partial n}|_{\partial F}=0,
\end{equation*}
\begin{equation}\label{p(s)}
 P(s)=\left(p_{ij}(s)\right),\; i,j=1,2,\ldots,m.
\end{equation}
\par It can be easily verified by the method of separation of variables that the eigenvalues and the corresponding eigenfunctions of $L(P(s))$, indexed by the pairs $(j,\beta) \in {\mathbb{Z}}\times \Gamma_{\delta}$,  are \\$\lambda_{j,\beta}=\lambda_j+\rvert \beta \rvert^2$\;\; and\;\; $\chi_{j,\beta}(x)=u_{\beta}(x) \cdot \varphi _{j}(s)=\left(u_{\beta(x)}\varphi_{j1}, u_{\beta(x)}\varphi_{j2}, \ldots, u_{\beta(x)}\varphi_{jm}\right)$, respectively, where $\beta\in \Gamma_{\delta}$,\; $\lambda_j$ is the eigenvalue and $\varphi_j(s)=\left(\varphi_{j,1}(s), \varphi_{j,2}(s), \ldots , \varphi_{j,m}(s) \right)$ is the corresponding eigenfunction of the operator $T(P(s))$ defined by the differential expression
\begin{equation}\label{ps}
T(P(s))Y=-\left\lvert\frac{\pi}{a_i}\right\lvert^2Y^{''}+P(s)Y
\end{equation}
and the boundary condition
\begin{equation}\label{y(0)}
Y^{'}(0)=Y^{'}(\pi)=0.
\end{equation}
\par The eigenvalues of the operator $T(0)$, defined by \eqref{ps} when $P(s)=0$\; and the boundary condition \eqref{y(0)},\; are\;  $\lvert n\delta \lvert^2=\lvert\frac{n \pi}{a_i}\lvert^2$ with the corresponding eigenspace $E_n=span\left\{C_{n,1}(s), C_{n,2}(s), \ldots , C_{n,m}(s)\right\} $, where $C_{n,i}(s)=(0, \ldots , \cos ns, \ldots , 0)$, $n\in {\mathbb{Z}^{+}\cup \{0\}}.$ It is well known that the eigenvalue $\lambda_j$ of $T(P(s))$ satisfying $\lvert\lambda_j- \lvert j \delta \lvert^2\lvert < \sup P(s)$, satisfies the following relation
\begin{equation}\label{eigenvalue}
\lambda_j=\lvert j \delta \lvert^2+O\left(\frac{1}{\lvert j \delta \lvert}\right).
\end{equation}
 \par By the above equation, the eigenvalue $\lvert \gamma \lvert^2=\lvert \beta \lvert^2+\lvert j \delta \lvert^2$ of $L(0)$ corresponds to the eigenvalue $\lvert \beta\lvert^2+\lambda_j$ of $L(P(s)).$
\par Note that, we denote the inner product in $L_2^    m(F)$ by $\langle\cdot,\cdot \rangle$ which is defined by using the inner product $(\cdot,\cdot)$ in $L_2(F)$ as follows:
\begin{align}\label{inner}
f(x)=(f_1(x),\ldots,f_m(x)),\;g(x)=(g_1(x),\ldots,g_m(x))\in L_2^m(F) \;\Rightarrow \;\langle f,g\rangle=(f_1,g_1)+\ldots+(f_m,g_m),
\end{align}
for $x\in \mathbb{R}^d$, $d\geq 1$. Also for any $f\in L_2^m[0,\pi]$, since $\{C_{n,i}\}_{n\in \mathbb{Z}^{+} \cup \{0\},\;i=1,2,\ldots,m}$ is a complete system, by \eqref{inner} we have the decomposition
\begin{align}\nonumber
f(s)=& \sum_{n\in \mathbb{Z}^{+} \cup \{0\}}\sum_{i=1}^{m}\frac{2}{\pi}\left\langle f (s), C_{n,i}(s)\right \rangle C_{n,i}(s)\\ =& \left( \sum_{n\in \mathbb{Z}^{+} \cup \{0\}} \frac{2}{\pi}\left( f_1(s), \cos ns\right ) \cos ns,\ldots, \sum_{n\in \mathbb{Z}^{+} \cup \{0\}} \frac{2}{\pi}\left( f_m(s), \cos ns\right ) \cos ns \right).
\end{align}
On the other hand, by equivalence of the decompositions \eqref{q1} and \eqref{q2} $\left( q(x)=q(s)\;\in L_2^m[0,\pi]\right.,$ $\left.  \text{when} \;\;d=1 \right)$, it is convenient to use the decomposition
$$f(s)=\sum_{n\in \mathbb{Z}}\sum_{i=1}^{m}\frac{1}{\pi}\left\langle f (s), C_{n,i}(s)\right \rangle C_{n,i}(s). $$
In the sequel, for the sake of simplicity, we use the brief notation $\left\langle f (s), C_{n,i}(s)\right \rangle$ instead of $\frac{1}{\pi}\left\langle f (s), C_{n,i}(s)\right \rangle$, since the constants which do not depend on $\rho$ are inessential in our calculations.
\par The system of eigenfunctions $\left\{\chi_{j,\beta}\right\}_{j,\beta}$ is complete in $L_2^m(F).$ Indeed; suppose that there exists a non-zero function $f(x)\in L_2^m(F)$ which is orthogonal to each $\chi_{j,\beta}$, $j\in \boldsymbol{Z}$, $\beta \in \Gamma_{\delta}.$ Since $C_{n,i}$\;, $i=1,2,\ldots,m$ can be decomposed by $\varphi_j,$ by \eqref{gamma}, and the definition of $\chi_{j,\beta},$ the function\; $\Phi_{i,\gamma}=u_{\beta}(x)\cdot C_{n,i}$\;, $i=1,2,\ldots,m$ can be decomposed by the system  $\left\{\chi_{j,\beta}\right\}_{j\in \boldsymbol{Z},\beta \in \Gamma_{\delta}}$. Thus, the assumption $\left \langle\chi_{j,\beta}(x)\;,\;f(x)\right\rangle=0$ for $j\in \boldsymbol{Z},\;\beta \in \Gamma_{\delta}$
implies that $\left\langle f(x)\;,\;\phi_{i,\gamma} \right\rangle=0$, $\forall \gamma\in \frac{\Gamma}{2}$ and $i=1,2,\ldots,m,$ which contradicts to the fact that $\left\{\Phi_{i,\gamma}(x)\right\}_{\gamma\in \frac{\Gamma}{2}\;,\;i=1,\ldots,m}$ is a basis for $L_2^m(F).$
 \par To prove the asymptotic formulas,
 we use the binding formula
 \begin{equation}\label{binding}
 \left(\Lambda_n -\lambda_{j,\beta} \right)\left\langle \psi_N \;,\; \chi_{j,\beta}\right\rangle=\left\langle \psi_N, \left(V-P\right) \chi_{j ,\beta}\right\rangle,
 \end{equation}
 for the eigenvalue, eigenfunction pairs $\Lambda_{N}$, $\Psi_{N}(x)$
 and $\lambda_{j,\beta},\chi_{j,\beta}$ of the operators
 $L(V)$ and $L(P(s))$, respectively. The formula \eqref{binding} can be
 obtained  by multiplying the equation
 \\$L(V)\Psi_{N}(x)=\Lambda_{N}\Psi_{N}(x)$ by $\chi_{j,\beta}$ and
 using the facts that $L(P(s))$ is self-adjoint and \\
 $L(P(s))\chi_{j,\beta}=\lambda_{j,\beta}\; \chi_{j,\beta}$.
 \par Now our aim is to decompose $\left( V-P \right) \chi_{j,\beta}$ with respect to the basis $\left\{\chi_{j^{'},\beta^{'}}\right\}_{j^{'}\in\boldsymbol{Z}, \beta^{'}\in \Gamma_{\delta}}$.
 \par By \eqref{vij} and \eqref{decomp}, we have
 \begin{equation}\label{vij-pij}
 v_{ij}(x)-p_{ij}(s)=\sum_{(\beta_1,n_1)\in \Gamma^{'}(\rho^{\alpha})}d_{ij}(\beta_1,n_1)\cos n_1 s\; u_{\beta_1}(x)+O(\rho^{-p\alpha}),
 \end{equation}
 where $$\Gamma^{'}(\rho^{\alpha})=\left\{(\beta_1,n_1):\beta_1 \in \Gamma_{\delta} \backslash \{0\},n_1\in \boldsymbol{Z},n_1\delta+\beta_1\in \Gamma(\rho^{\alpha})\right\}$$ and \\$$d_{ij}(\beta_1,n_1)=\frac{1}{\mu(F)}\int \limits_{F}v_{ij}(x)\cos n_1s\; u_{\beta_1}(x)dx.$$
  For $(\beta_1,n_1)\in \Gamma^{'}(p\rho^{\alpha}),$ we have $\lvert n_1\delta+\beta_1\lvert<p\rho^{\alpha}$ and since $\beta_1$ is orthogonal to $\delta$,
  \begin{equation}\label{single2}
  \lvert \beta_1
    \lvert<p\rho^{\alpha},\;\;\lvert n_1 \lvert<p\rho^{\alpha}\;\;\lvert n_1 \lvert<\frac{1}{2}r_1,
  \end{equation}
  (see \eqref{single})
  \par Clearly (see equation (22) in \cite{Karakilic2}), we have, for all $i,j=1,2,\ldots,m,$
 \begin{equation}\label{sum}
 \sum_{\left( \beta_1,n_1 \right)\in \Gamma^{'}(\rho^{\alpha})}d_{ij}\left(  \beta_1,n_1 \right)\left( \cos n_1s  \right)u_{\beta_1}(x)u_{\beta}(x)=\sum_{\left( \beta_1,n_1 \right)\in \Gamma^{'}(\rho^{\alpha})}d_{ij}\left(  \beta_1,n_1 \right)\left( \cos n_1s  \right)u_{\beta_1+\beta}(x),
 \end{equation}
 for all \;$\beta\in \Gamma_{\delta}$ satisfying $\left\lvert\beta^k \right\lvert >\frac{1}{3}\rho^{\alpha_1}$, $\forall k:e_k \neq \delta.$
\par By using the definition of $\chi_{j,\beta}$, $P(s)$, the decompositions \eqref{vij-pij} and \eqref{sum}, we have
\begin{align}\label{1decom}
&\left( V-P \right) \chi_{j,\beta}=\nonumber \\ \nonumber &\sum_{\left( \beta_1,n_1 \right)\in \Gamma^{'}(\rho^{\alpha})}\sum_{k=1}^{m}\left( d_{1k}\left(  \beta_1,n_1 \right)\left( \cos n_1s  \right)\varphi_{j,k}(s)u_{\beta+\beta_{1}},\ldots,d_{mk}\left(  \beta_1,n_1 \right)\left( \cos n_1s  \right)\varphi_{j,k}(s)u_{\beta+\beta_{1}} \right)\\ &\quad \quad \quad \quad \quad\quad\quad\quad\quad
\quad\quad \quad \quad \quad \quad\quad\quad\quad\quad
\quad\quad \quad \quad \quad \quad\quad\quad\quad\quad\quad\quad\quad
\quad
+O\left( \rho^{-p\alpha} \right).
\end{align}
\hspace{.35cm}Now we consider the following decompositions:
\begin{equation}\label{cos}
\varphi_{j,k}(s)=\sum_{n\in \boldsymbol{Z}}\left(\varphi_{j,k}, \cos ns\right)\; \cos ns,\end{equation}
\begin{eqnarray}\label{cosn1s}
\cos n_1 s\; \varphi _{j,k}(s)&=& \sum_{n\in \boldsymbol{Z}}\left(\varphi_{j,k}, \cos ns\right). \;\cos n_1s. \cos ns\nonumber\\
&=&\sum_{n\in \boldsymbol{Z}}\left(\varphi_{j,k}, \cos ns\right).\frac{1}{2}[\cos(n_1+n)s+\cos(n_1-n)s]\nonumber\\
&=&\sum_{n\in \boldsymbol{Z}}\left(\varphi_{j,k}, \cos ns\right).\cos(n_1+n)s,
\end{eqnarray}
for each $j\in Z$, $k=1,2,\ldots, m.$
\par On the other hand; the decomposition of $\varphi_{j}(s)=\left(\varphi_{j,1}(s),\ldots,\varphi _{j,m}(s)\right)$ with respect to the basis $\left\{C_{n,i}(s)=\left(0,0,\ldots,\cos ns,0,\ldots,0 \right)\right\}_{n\in \boldsymbol{Z}, i=1,2,\ldots,m}$ is given by
\begin{eqnarray}\label{phi}
\varphi_{j}(s)&=& (\varphi_{j,1}, \varphi_{j,2},\ldots, \varphi_{j,m})\nonumber \\ &=& \sum_{n\in \boldsymbol{Z}}\sum_{i=1}^{m}\langle\;\varphi_{j}(s),C_{n,i}(s)\;\rangle C_{n,i}(s) \nonumber\\
&=&\left(\sum_{n\in \boldsymbol{Z}}\langle\;\varphi_{j}(s),C_{n,1}(s)\;\rangle \cos ns,\ldots, \sum_{n\in \boldsymbol{Z}}\langle\;\varphi_{j}(s),C_{n,m}(s)\;\rangle \cos ns\right).
\end{eqnarray}
Thus, \eqref{cos}, \eqref{cosn1s} and \eqref{phi}, gives
\begin{eqnarray}\label{cos1}
& \varphi _{j,k}(s)=\sum\limits_{n\in \boldsymbol{Z}}\langle\;\varphi_{j}(s),C_{n,k}(s)\;\rangle\;\cos ns\\ \nonumber& \cos n_1 s\; \varphi _{j,k}(s)=\sum\limits_{n\in \boldsymbol{Z}}\langle\;\varphi_{j}(s),C_{n,k}(s)\;\rangle \;\cos(n+n_1)s.
\end{eqnarray}

\begin{lemma}
Let $r$ be a number no less than $r_1$ $(r\geq r_1)$ and $j,n$ be integers satisfying $\lvert j \lvert +1 < r$, $\lvert n\lvert\geq 2r$. Then
\begin{equation}\label{lemma1.1}
\langle\;\varphi_j(s)\;,\;C_{n,i}(s)\;\rangle=O\left( \rho^{-(l-1)\alpha} \right)\;, \forall i=1,2,\ldots,m\end{equation}
and
\begin{equation}\label{lemma1.2}
\varphi_{j}(s)=\sum_{\lvert n\lvert<2r}\sum_{i=1}^{m}\left\langle\;\varphi_j(s)\;,\;C_{n,i}(s)\;\right\rangle \;C_{n,i}(s)+O\left(\rho^{-(l-2)\alpha}\right).
\end{equation}
\end{lemma}
\begin{proof}
We use the following binding formula for $T(0)$ and $T(P(s))$
\begin{equation}\label{binding1}
\left( \lambda_j- \lvert n\delta \lvert ^2 \right)\left\langle\varphi_{j}(s), C_{n,k}(s) \right\rangle=\langle\varphi_{j}(s), P(s)C_{n,k}\rangle
\end{equation}
and the obvious decomposition, which can be obtained by definition of $P(s)$ and \eqref{decomp},
 \begin{align}\nonumber
P(s)C_{n,k}(s)&=\left(\sum_{\lvert n_1 \delta \lvert <\frac{\lvert n\delta\lvert}{2l}}p_{1kn_1}\cos n_1s \cos ns,\ldots, \sum_{\lvert n_1 \delta \lvert <\frac{\lvert n\delta \lvert}{2l}}p_{mkn_1}\cos n_1s \cos ns \right)+O\left(\lvert n\delta \lvert ^{-(l-1)}\right)\\ \nonumber
&=\left(\sum_{\lvert n_1 \delta \lvert <\frac{\lvert n\delta \lvert}{2l}}p_{1kn_1}\cos (n-n_1)s,\ldots, \sum_{\lvert n_1 \delta \lvert <\frac{\lvert n\delta \lvert}{2l}}p_{mkn_1}\cos (n-n_1)s\right)+O\left(\lvert n\delta \lvert ^{-(l-1)}\right)\\ \label{2}
&=\sum_{t=1}^{m}\sum_{\lvert n_1 \delta \lvert <\frac{\lvert n\delta \lvert}{2l}}p_{tkn_1}C_{n-n_1,k}(s)+O\left(\lvert n\delta \lvert ^{-(l-1)}\right).
\end{align}
Putting above equation \eqref{2} into \eqref{binding1}, we get
\begin{eqnarray}\nonumber
\left( \lambda_j- \lvert n\delta \lvert ^2 \right)\langle\varphi_{j}(s), C_{n,k}(s) \rangle&=&\langle\varphi_{j}(s), \sum_{t_1=1}^{m}\sum_{\lvert n_1 \delta \lvert <\frac{\lvert n\delta \lvert}{2l}}p_{t_1kn_1} C_{n-n_1,k}\rangle+O\left(\lvert n\delta \lvert ^{-(l-1)}\right)\\ \label{**}
&=&\sum_{t_1=1}^{m}\sum_{\lvert n_1 \delta \lvert <\frac{\lvert n\delta \lvert}{2l}}p_{t_1kn_1}\langle\varphi_{j}(s), C_{n-n_1,k}(s) \rangle +O\left(\lvert n\delta \lvert ^{-(l-1)}\right)
\end{eqnarray}
By assumption $\lvert n \lvert\geq 2r$ and $\lvert j \lvert +1<r$, thus if $\lvert n_1 \delta \lvert<\frac{\lvert n \delta\lvert}{2l}$ then $\lvert\lvert(n-n_1)\delta\lvert^2-\lvert j\lvert\lvert>\frac{\lvert n \lvert}{5}$ which together with \eqref{eigenvalue} imply $\lvert \lambda_j-\lvert(n-n_1)\delta\lvert^2\lvert>c\lvert n\delta \lvert$. So that in \eqref{binding1} if we substitute $(n-n_1)\delta$ instead of $n\delta$, we get
\begin{equation}\label{1*}
 \langle\varphi_j(s),C_{n-n_1,k}(s)\rangle=\frac{\langle\varphi_j(s),P(s)C_{n-n_1,k}\rangle} {\lambda_j-\lvert (n-n_1)\delta\lvert^2}
\end{equation}
Now using \eqref{1*} in \eqref{**}, we get
\begin{equation*}\label{it}
\left( \lambda_j- \lvert n\delta \lvert ^2 \right)\langle\varphi_{j}(s), C_{n,k}(s) \rangle=\sum_{t_1=1}^{m}\sum_{\lvert n_1\delta \lvert <\frac{\lvert n\delta \lvert}{2l}}\frac{p_{t_1kn_1}\langle\varphi_{j}(s), P(s) \;C_{n-n_1,k}(s) \rangle}{ \left( \lambda_j- \lvert (n-n_1)\delta \lvert ^2 \right)  } +O\left(\lvert n\delta \lvert ^{-(l-1)}\right).
\end{equation*}
Again putting \eqref{2} into the last equation, we obtain
\begin{align}
&\left(\lambda_j- \lvert n\delta \lvert ^2 \right)\langle\varphi_{j}(s), C_{n,k}(s) \rangle\nonumber \\ &=\sum_{t_1=1}^{m}\sum_{\lvert n_1 \delta \lvert <\frac{\lvert n\delta \lvert}{2l}}\frac{p_{t_1kn_1}\langle\varphi_{j}(s),\sum_{t_2=1}^{m} \sum_{\lvert n_2 \delta \lvert <\frac{\lvert n\delta \lvert}{2l}}p_{t_2kn_2} \;C_{n-n_1-n_2,k}(s) \rangle}{ \left( \lambda_j- \lvert (n-n_1)\delta \lvert ^2 \right)  } +O\left(\lvert n\delta \lvert ^{-(l-1)}\right)\nonumber\\
&=\sum_{t_1,t_2=1}^{m}\sum_{\lvert n_1 \delta \lvert< \frac{\lvert n\delta \lvert}{2l} \atop \lvert n_2 \delta \lvert< \frac{\lvert n\delta\lvert}{2l} }\frac{p_{t_1kn_1}p_{t_2kn_2}\langle\varphi_{j}(s), C_{n-n_1-n_2,k}(s) \rangle +O\left(\lvert n\delta \lvert ^{-(l-1)}\right)}{ \left( \lambda_j- \lvert (n-n_1)\delta \lvert ^2 \right)}.
\end{align}

In this way, iterating $p_1=[\frac{l}{2}]$ times and dividing both sides of the obtained equation by $\lambda_j- \lvert n\delta \lvert ^2$, we have
\begin{equation}\label{afteriteration}
\langle \varphi_j(s)C_{n,k}(s)\rangle=\sum_{t_1,t_2,\ldots,t_{p_1}=1}^m \sum_{\substack{\lvert n_1\delta \lvert<\frac{\lvert n\delta \lvert}{2l}\\\lvert n_2\delta \lvert<\frac{\lvert n\delta \lvert}{2l}\\ \vdots\\\lvert n_{p_1}\delta \lvert<\frac{\lvert n\delta \lvert}{2l}}}\frac{p_{t_1kn_1}p_{t_2k n_2}\ldots p_{t_{p_1}kn_{p_1}}\langle\varphi_{j}, C_{n-n_1-\ldots-n_{p_1},k} \rangle}{\Pi_{s=0}^{p_1-1} \left(\lambda_{j}- \lvert (n-n_1-\ldots-n_s)\delta \lvert ^2\right)}+O(\lvert n \delta \lvert^{-(l-1)})
\end{equation}
where the integers $n,n_1,\ldots, n_{p_1}$ satisfy the conditions
$$\lvert n_s \lvert<\frac{\lvert n \lvert}{2l},\;\;\; s=1,\ldots, p_1,\;\;\; \lvert j \lvert+1<\frac{\lvert n\lvert}{2}.  $$
These conditions and the assumptions $\lvert n \lvert>2r$, $\lvert j \lvert+1<r$ imply that
$$ \lvert \lvert n-n_1-\ldots-n_s\lvert-\lvert j \lvert>\frac{\lvert n \lvert}{5},\;\;s=0,1,2,\ldots,p_1.$$ This together with \eqref{eigenvalue}, give
\begin{equation}\label{condition}
\frac{1}{\lvert \lambda_j-\lvert(n-n_1-\ldots-n_s)\delta \lvert^2}=\frac{1}{\left \lvert \lvert j\delta \lvert^2+O\left(\frac{1}{\lvert j\delta \lvert}\right)- \lvert\left(n-n_1-\ldots -n_s\right)\delta\lvert^2\right\lvert}=O\left(\lvert n\delta \lvert^{-2}\right)
\end{equation}
for $s=0,\ldots,{p_1}-1.$
Hence by \eqref{afteriteration}, \eqref{condition} and \eqref{decompsum}, we have
$$  \langle \varphi_j(s),C_{n,k}(s)\rangle=O\left(\lvert n\delta \lvert^{-(l-1)}\right).$$
Since $\lvert n\delta \lvert\geq2r\geq r_1>2\rho^{\alpha}$, $O\left(\lvert n\delta \lvert^{-(l-1)}\right)=O(\rho^{-(l-1)\alpha})$ from which we get the proof of \eqref{lemma1.1}. \\

To prove \eqref{lemma1.2}, we write the Fourier series of $\varphi_j(s)$ with respect to the basis $\{C_{n,1}(s),\ldots,C_{n,m}(s)\}_{n\in \mathbb{Z}}$ as follows:
\begin{align*}
\varphi_j(s)&=\sum_{n\in \boldsymbol{Z}}\langle\varphi_j(s), C_{n,k}(s) \rangle  C_{n,k}(s)\\&=\sum_{\lvert n \lvert<2r}\langle\varphi_j(s), C_{n,k}(s) \rangle  C_{n,k}(s)+\sum_{\lvert n \lvert\geqslant 2r}\langle\varphi_j(s), C_{n,k}(s) \rangle  C_{n,k}(s),
\end{align*}
From which together with \eqref{lemma1.1}, we get \eqref{lemma1.2}.
\end{proof}
Using the first relation \eqref{lemma1.1} in Lemma $1$  and \eqref{cos1}, we also have
\begin{equation}\label{lemma1}
\cos n_1 s\; \varphi_{j,k}(s)=\sum_{\lvert n\lvert<2r}\langle\;\varphi_j(s)\;,\;C_{n,k}(s)\;\rangle\; \cos(n+n_1)s+O\left(\rho^{-(l-2)\alpha}\right).
\end{equation}
Putting this last relation \eqref{lemma1} into \eqref{1decom}, we get
\begin{align}\label{decom2}
\nonumber&(V-P)\chi_{j , \beta}=\\ \nonumber &\sum_{\left( \beta_1,n_1 \right)\in \Gamma^{'}(\rho^{\alpha})}\sum_{\lvert n\lvert<2r}\sum_{k=1}^{m}\left( d_{1k}\left(  \beta_1,n_1 \right)\langle\;\varphi_j(s)\;,\;C_{n,k}(s)\;\rangle\; \cos(n+n_1)su_{\beta+\beta_{1}},\ldots,\right.\nonumber\\
&\left. d_{mk}\left(  \beta_1,n_1 \right)\langle\;\varphi_j(s)\;,\;C_{n,k}(s)\;\rangle\; \cos(n+n_1)su_{\beta+\beta_{1}} \right)+O\left( \rho^{-p\alpha} \right).
\end{align}
$\left(\text{ note that } p=(l-d), d\geq 2 \Rightarrow \frac{1}{\rho(l-2)}<\frac{1}{\rho^{p\alpha}}.\text{ Hence  }  O\left( \rho^{-p\alpha} \right)+O\left(\rho^{-(l-2)\alpha}\right)=O\left( \rho^{-p\alpha}\right).\right)$
\par Now, in order to decompose $(V-P)\chi_{j , \beta}$ with respect to $\left\{\chi_{j+j_1^{'} , \beta_1^{'}}\right\}$ we consider the inner product $\langle (V-P)\chi_{j , \beta},\chi_{j+j_1^{'} , \beta_1^{'}}\rangle$, that is, by the definition of $\chi_{j+j_1^{'} , \beta_1^{'}}$ and \eqref{decom2}, the inner products \\$( \cos(n+n_1)s\;u_{\beta+\beta_1} \;,\;\varphi_{j+j_1^{'},t}(s)\;u_{\beta_1^{'}})$, $t=1,2,\ldots, m$. Using the decomposition \eqref{cos1}, instead of $j$, we substitute $j+j_1^{'}$ to get
\begin{eqnarray*}
\left( \cos(n+n_1)s\;u_{\beta+\beta_1} \;,\;\varphi_{j+j_1^{'},t}(s)\;u_{\beta_1^{'}}\right) &=&\left(\cos(n+n_1)s\;u_{\beta+\beta_1}\;,\; \sum_{n^{'}\in \boldsymbol{Z}}\langle\;\varphi_{j+j_1^{'}}(s)\;,\;C_{n^{'},t}(s)\;\rangle\;\cos n^{'}s\;u_{\beta_1^{'}}\right) \\
&=& \sum_{n^{'}\in \boldsymbol{Z}} \langle\;\overline{\varphi_{j+j_1^{'}}(s)\;,\;C_{n^{'},t}(s)}\;\rangle\left(\cos(n+n_1)s\;u_{\beta+\beta_1},\cos n^{'}s\;u_{\beta_1^{'}}\right)   .
\end{eqnarray*}
Note that if $\beta_1^{'}\neq \beta+\beta_1$ or $n^{'}\neq n+n_1$ then $( \cos(n+n_1)s\;u_{\beta+\beta_1} \;,\;\cos n^{'}s\;u_{\beta_1^{'}})=0$. Thus,
\begin{displaymath}
        \left(  \cos(n+n_1)s\;u_{\beta+\beta_1} \;,\;\varphi_{j+j_1^{'},t}(s)\;u_{\beta_1^{'}} \right)=\left\{ \begin{array}{ll}0 & \textrm{,\;\;if \;\;$\beta_1^{'}\neq\beta+\beta_1$\;\; or\;\; $n{'}\neq n+n_1$} \\
        \langle\;\overline{\varphi_{j+j_1^{'}}(s)\;,\;C_{n+n_1,t}(s)}\;\rangle & \textrm{,\;\;otherwise.}
        \end{array}\right.
        \end{displaymath}
Using the last equality and \eqref{decom2}, we get
\begin{multline}\label{decom3}
V-P)\chi_{j , \beta}=\sum_{j_1^{'}\in \boldsymbol{Z} \atop \left( \beta_1,n_1 \right)\in \Gamma^{'}(\rho^{\alpha})}\left(\sum_{\lvert n\lvert<2r}\sum_{k=1}^{m}\sum_{i=1}^{m}d_{ik}\left(n_1,\beta_1\right)\langle \varphi_{j},C_{n,k} \rangle \langle \overline{\varphi_{j+j_1^{'}},C_{n+n_1,i}}  \rangle \right)\chi_{j+j_1^{'},\beta+\beta_1}\\+O(\rho^{-p\alpha}).
\end{multline}
\begin{lemma}
Let $r$ be a number no less than $r_1$ $(r\geq r_1)$, $j,n$ and $n_1$be integers satisfying  $\lvert n\lvert<2r$, $\lvert n_1 \lvert<\frac{1}{2}r_1$ and $\lvert j \lvert+1<r,$ then
\begin{equation*}
\sum_{j_1\in \boldsymbol{Z} \atop \lvert j_1 \lvert \geqslant 6r} \langle \varphi_{j+j_1},C_{n,i} \rangle=O\left(\rho^{-(l-2)\alpha}\right), \forall i=1,2,\ldots,m.
\end{equation*}
\end{lemma}
\begin{proof}
By the binding formula \eqref{binding1} for $T(0)$ and $T(P(s))$ we have
\begin{equation}\label{*}
\left( \lambda_{j+j_1}- \lvert (n+n_1)\delta \lvert ^2 \right)\langle\varphi_{j+j_1}, C_{n+n_1,k} \rangle=\langle\varphi_{j+j_1}, P(s)C_{n+n_1,k}\rangle.
\end{equation}
If $\lvert j_1\lvert\geq 6r$ then the assumptions of this lemma imply  $\lvert \lvert j+j_1 \lvert-\lvert n+n_1 \lvert \lvert>\frac{r}{2}$. Thus, using \eqref{*} and the fact that $\lambda_{j+j_1}=\lvert (j+j_1)\delta \lvert^2+O\left(\frac{1}{\lvert (j+j_1)\delta \lvert}\right)$, we get

 $$\lvert \sum_{j_1  \geqslant 6r} \langle\varphi_{j+j_1}, C_{n+n_1,k} \rangle \lvert=\lvert  \sum_{j_1\geqslant 6r}\frac{\langle\varphi_{j+j_1}, P(s)C_{n+n_1,k} \rangle}{\lambda_{j+j_1}- \lvert (n+n_1)\delta \lvert ^2}\lvert.$$
\par Using the decomposition of $p_{tk}(s)=\left(\sum_{\lvert l_1\delta \lvert<\lvert r\delta \lvert}v_{tk, l_1\delta}\cos l_1s\right)+O(\lvert r \delta \lvert^{-(l-1)})$ and iterating the obtained formula $p_1=[\frac{l}{2}]$ times as in the proof of Lemma 1, we get
\begin{equation}\label{6r}
\lvert \sum_{\lvert j_1 \lvert \geqslant 6r} \langle\varphi_{j+j_1}, C_{n+n_1,k} \rangle \lvert=\lvert \sum_{j_1  \geqslant 6r}\sum_{\substack{\lvert l_1 \delta \lvert< \lvert r\delta \lvert\\ \lvert l_2 \delta \lvert< \lvert r\delta\lvert\\\vdots\\\lvert l_p \delta \lvert< \lvert r\delta \lvert}  }\sum _{t_1,t_2,\ldots,
t_p=1}^{m}\frac{v_{t_1k,l_1\delta}\ldots v_{t_pk,l_p\delta}\langle \varphi_{j^{'}},C_{n+n_1-l_1-\ldots-l_k}\rangle}{\prod _{s=0}^{p-1}\lvert \lambda_{j+j_1}- (n+n_1-l_1-\ldots-l_s)\delta \lvert^2}
\end{equation}
 Since $\lvert n \lvert< 2r$ and $\lvert n_1 \lvert <\frac{1}{2}r_1<\frac{1}{2}r$,\; $\lvert n+n_1 \lvert<\frac{5r}{2}$. Also, \\
 $\lvert  n+n_1-l_1-\ldots -l_s \lvert<3r$\;\; \;\;\;and \; \;\;\;$ \frac{1}{\lvert \lambda_{j+j_1}-\lvert \left(  n+n_1-l_1-\ldots-l_s \right)\delta \lvert^2\lvert}=O\left(\lvert r \lvert^{-2}\right)$.
 Substituting this result into \eqref{6r} and using \eqref{decompsum}, we get the proof.
  \end{proof}
\par By Lemma 2, the equation \eqref{decom3} becomes;
\begin{align*}
(V-P)\chi_{j , \beta}&=O(\rho^{-p\alpha})+\\&\sum_{\lvert j_1^{'} \lvert < 6r \atop \left( \beta_1,n_1 \right)\in \Gamma^{'}(\rho^{\alpha})}\left(\sum_{\lvert n\lvert<2r}\sum_{k=1}^{m}\sum_{i=1}^{m}d_{ik}\left(n_1,\beta_1\right)\langle \varphi_{j},C_{n,k} \rangle \langle \overline{\varphi_{j+j_1{'}},C_{n+n_1,i}}  \rangle \right)\chi_{j+j_1^{'},\beta+\beta_1}\\&=O(\rho^{-p\alpha})+\\&\sum_{\lvert j_1 \lvert < 6r \atop \left( \beta_1,n_1 \right)\in \Gamma^{'}(\rho^{\alpha})}\left(\sum_{\lvert n\lvert<2r}\sum_{k=1}^{m}\sum_{i=1}^{m}d_{ik}\left(n_1,\beta_1\right)\langle \varphi_{j},C_{n,k} \rangle \langle \overline{\varphi_{j+j_1},C_{n+n_1,i}}  \rangle \right)\chi_{j+j_1,\beta+\beta_1},
\end{align*}
that is,
\begin{equation}\label{decomplast}
(V-P)\chi_{j , \beta}=\sum_{\left(\beta_1,j_1\right)\in \boldsymbol{Q}(\rho^{\alpha},6r)}A\left(j,\beta,j+j_1,\beta+\beta_1\right)\chi_{j+j_1,\beta+\beta_1}+O(\rho^{-p\alpha}),
\end{equation}
for every $j$ satisfying $\lvert j \lvert+1<r$, where
$$\boldsymbol{Q}(\rho^{\alpha},6r)=\left\{(j,\beta): \lvert j \delta \lvert <6r \; ,\; 0<\lvert \beta  \lvert < \rho^{\alpha} \right\}, $$
$$A\left(j,\beta,j+j_1,\beta+\beta_1\right)=\sum_{n_1:\left(n_1,\beta_1\right)\in \Gamma^{'}(\rho^{\alpha})}\left(\sum_{\lvert n\lvert<2r}\sum_{k=1}^{m}\sum_{i=1}^{m}d_{ik}\left(n_1,\beta_1\right)\langle \varphi_{j},C_{n,k} \rangle \langle \overline{\varphi_{j+j_1},C_{n+n_1,i}}  \rangle \right). $$
We need to prove that
\begin{equation}\label{absolute}
\sum_{\left(\beta_1,j_1\right)\in \boldsymbol{Q}(\rho^{\alpha},6r)}\left\lvert A\left(j,\beta,j+j_1,\beta+\beta_1\right)\right\lvert<c_3.
\end{equation}
By the definition of $A\left(j,\beta,j+j_1,\beta+\beta_1\right)$, $d_{ik}\left(n_1,\beta_1\right)$ and \eqref{decompsum}, we have
\begin{align}\nonumber
\sum_{\left(\beta_1,j_1\right)\in \boldsymbol{Q}(\rho^{\alpha},6r)} & \left\lvert A\left(j,\beta_1,j^{'},\beta+\beta_1\right)\right\lvert\\ \nonumber &\leq\sum_{n_1:\left(n_1,\beta_1\right)\in \Gamma^{'}(\rho^{\alpha})}\sum_{i,k=1}^{m}\left\lvert d_{ik}(n_1,\beta_1)\right\lvert\sum_{\lvert n\lvert<2r}\left\lvert \langle \varphi_{j},C_{n,k} \rangle\right \lvert\sum_{\lvert j_1\lvert<6r}\left\lvert \langle \overline{\varphi_{j+j_1},C_{n+n_1,i}}  \rangle \right \lvert\\  \label{A}&\leq c_4 \sum_{\lvert n\lvert<2r}\left\lvert \langle \varphi_{j},C_{n,k} \rangle\right \lvert\sum_{\lvert j_1\lvert<6r}\left\lvert \langle \varphi_{j+j_1},C_{n+n_1,i}  \rangle \right \lvert
\end{align}
Now we prove that
\begin{equation}\label{c67}
\sum_{n\in \mathbb{Z}}\left\lvert \langle \varphi_{j},C_{n,k} \rangle\right \lvert<c_5 \text{\;\;\;\;\;and\;\;\;\;}\sum_{j_1\in\mathbb{Z}}\left\lvert \langle \varphi_{j+j_1},C_{n+n_1,i}  \rangle \right \lvert<c_6
\end{equation}
For this, let $$A=\left\{n\in \mathbb{Z}\;\;|\;\;\lvert n\delta \lvert^2\in [\lambda_{j-1},\lambda_{j+1}]\;\;\right\}$$ and  $$B=\left\{j_1\in \mathbb{Z}\;\;|\;\;\lambda_{j+j_1}\in \left[\;\left \lvert (n+n_1)\delta\right\lvert^2-1,\;\;\left\lvert (n+n_1)\delta\;\right\lvert^2+1\;\right]\;\;\right\},$$ then it follows from \eqref{eigenvalue} that the number of elements in the sets $A$ and $B$ are less than $c_7$. So if we isolate the terms with $n\in A$ and $j_1\in B$ in the first and second summations of inequalities in \eqref{c67}, respectively, appliying  \eqref{binding1} to the other tems then using the facts
\begin{equation*}
\sum_{n\notin A}\frac{1}{\lvert\lambda_j-\lvert n\delta\lvert^2\lvert}<c_8,\;\;\;\;\sum_{j_1\notin \beta}\frac{1}{\lvert\lambda_{j+j_1}-\lvert (n+n_1)\delta\lvert^2\lvert}<c_9
\end{equation*}
we get \eqref{c67}, hence by \eqref{A},  \eqref{absolute} is proved.
\par The expressions \eqref{decomplast} and \eqref{binding} together
imply that
\begin{equation}\label{decom3}
\left(\Lambda_N -\lambda_{j^{'},\beta^{'}} \right)\langle \psi_N \;,\; \chi_{j^{'},\beta^{'}}\rangle=\sum_{\left(\beta_1,j_1\right)\in {\boldsymbol{Q}(\rho^{\alpha}6r)}}
A\left(j^{'},\beta^{'},j^{'}+j_1,\beta^{'}+\beta_1\right)\langle \psi_N \;,\; \chi_{j^{'}+j_1,\beta^{'}+\beta_1}\rangle+O(\rho^{-p\alpha}).
\end{equation}
If the condition (iterability condition for the triple $(N,j^{'},\beta^{'})$ )
\begin{align}\label{cond1}
\lvert \Lambda_{N}-\lambda_{j^{'},\beta^{'}} \lvert >c_{10}
\end{align}
holds then the formula \eqref{decom3} can be written in the following form
\begin{align}\label{decom4}
\langle \psi_N \;,\; \chi_{j^{'},\beta^{'}}\rangle=\sum_{\left(\beta_1,j_1\right)\in {\boldsymbol{Q}(\rho^{\alpha}6r)}}
\frac{A\left(j^{'},\beta^{'},j^{'}+j_1,\beta^{'}+\beta_1\right)\langle \psi_N \;,\; \chi_{j^{'}+j_1,\beta^{'}+\beta_1}\rangle}{\Lambda_N -\lambda_{j^{'},\beta^{'}}}+O(\rho^{-p\alpha}).
\end{align}
Using \eqref{decom3} and \eqref{decom4}, we are going to find $\Lambda_{N}$ which is close to $\lambda_{j,\beta}$, where $\lvert j \lvert +1<r_1.$For this, first in \eqref{decom3} instead of $j^{'},\beta^{'}$, taking $j,\beta$, hence instead of $r$ taking $r_1$, we get
\begin{align}\label{decom5}
\left(\Lambda_N -\lambda_{j,\beta}\right)\langle \psi_N \;,\; \chi_{j,\beta}\rangle=\sum_{\left(\beta_1,j_1\right)\in {\boldsymbol{Q}(\rho^{\alpha},6r_1)}}
A\left(j,\beta,j+j_1,\beta+\beta_1\right)\langle \psi_N \;,\; \chi_{j+j_1,\beta+\beta_1}\rangle+O(\rho^{-p\alpha}).
\end{align}
  To iterate it by using \eqref{decom4} for $j^{'}=j+j_1$ and $\beta^{'}=\beta+\beta_1$, we will prove that there is a number $N$ such that
\begin{equation}\label{cond4}
\lvert \Lambda_{N}-\lambda_{j+j_1,\beta+\beta_1}\lvert>\frac{1}{2}\rho^{\alpha_2},
\end{equation}
where $\lvert j+j_1 \lvert<7r_1\equiv r_2$, since $\lambda_{j,\beta}$ and $\lvert j_1 \lvert<6r_1.$ Then $(j+j_1,\beta+\beta_1)$ satisfies \eqref{cond1}. This means that, in formula \eqref{decom2}, the pair $(j^{'},\beta^{'})$ can be replaced by the pair $(j+j_1,\beta+\beta_1)$. Then, \eqref{decom2} instead of $r$ taking $r_2$, we get
\begin{align*}
&\langle \psi_N \;,\; \chi_{j+j_1,\beta+\beta_1}\rangle=\\&O(\rho^{-p\alpha})+\sum_{\left(\beta_2,j_2\right)\in {\boldsymbol{Q}(\rho^{\alpha},6r_2)}}
\frac{A\left(j+j_1,\beta+\beta_1,j+j_1+j_2,\beta+\beta_1+\beta_2\right)\langle \psi_N \;,\; \chi_{j+j_1+j_2,\beta+\beta_1
+\beta_2}\rangle}{\Lambda_N -\lambda_{j+j_1,\beta+\beta_1}}.
\end{align*}
Putting the above formula into \eqref{decom5}, we obtain
\begin{align}\label{decom8}
\left(\Lambda_{N}-\lambda_{j,\beta}\right)c(N,j,\beta)=O(\rho^{-p\alpha})+\sum_{\left(\beta_1,j_1\right)\in {\boldsymbol{F}(\rho^{\alpha},6r_1)}\atop\left(\beta_2,j_2\right)\in {\boldsymbol{F}(\rho^{\alpha},6r_2)}}\frac{A\left(j,\beta,j^{1},\beta^{1}\right)A\left(j^1,\beta^1,j^2,\beta^2\right)c(N,j^2,\beta^2)}{\Lambda_{N}-\lambda_{j^1,\beta^1}}
\end{align}
where $c(N,j,\beta)=\langle\psi_{N},\chi_{j,\beta}\rangle$, $j^k=j+j_1+j_2+\ldots+j^k$ and $\beta^k=\beta+\beta_1+\beta_2+\ldots+\beta_k$. Thus, we are going to find a number $N$ such that $c(N,j,\beta)$ is not too small and the condition \eqref{cond4} is satisfied.
\begin{lemma}\label{lemma 3}
\begin{itemize}
\item[(a)]
Suppose $h_1(x), h_2(x),\ldots,h_{p_2}(x)\in L_2^m (F)$ where $p_2=[\frac{d}{2\alpha_2}]+1$. Then for every eigenvalue $\lambda_{j,\beta}$ of the operator $L(P(s)),$ there exists an eigenvalue $\Lambda_{N}$ of $L(V)$ satisfying
\begin{itemize}
\item[(i)]
$\lvert \Lambda_{N}-\lambda_{j,\beta}\lvert<2M$, where $M=\rVert V \rVert$,
\item[(ii)]
$\lvert c(N,j,\beta) \lvert >\rho^{-q\alpha},$ where $q\alpha=[\frac{d}{2\alpha}+2]\alpha,$
\item[(iii)]
$\lvert c(N,j,\beta) \lvert ^2>\frac{1}{2p_2}\sum_{i=1}^{p_2} \lvert \langle\psi_{N},\frac{h_i}{\rVert h_i \rVert}\rangle \lvert^2>\frac{1}{2p_2}\lvert \langle\psi_{N},\frac{h_i}{\rVert h_i \rVert}\rangle\lvert^2,$ $\forall i=1,2,\dots,p_2.$
\end{itemize}
\item[(b)]
Let $\gamma=\beta+j\delta\in V_{\delta}^{'}(\alpha)$ and $(\beta_1,j_1)\in Q(\rho^{\alpha},6r_1), (\beta_k,j_k)\in Q(\rho^{\alpha},6r_k),$ where $r_k=7r_{k-1}$ for $k=2,3,\ldots, p$. Then for $k=1,2,3,\ldots,p,$ we have
\begin{align}\label{cond6}
\lvert \lambda_{j,\beta}-\lambda_{j^k,\beta^k} \lvert>\frac{3}{5}\rho^{\alpha_2},\;\; \forall \beta^k \neq \beta.
\end{align}
\end{itemize}
\end{lemma}
\begin{proof}
 \begin{itemize}
 \item[(a)]
 Let $A,B,C$ be the set of indexes $N$ satisfying (i),
(ii), (iii), respectively. Using the binding formula
\eqref{binding} for $L(V)$ and $L(P(s))$ and the Bessel's
inequality, we get
\begin{eqnarray*}
\sum_{N \notin A}|c(N,j,\beta)|^2=\sum_{N \notin
A}\left|\frac{(\psi_N,(V-P)\chi_{j,\beta}) }{\Lambda_N
-\lambda_{j,\beta}}\right|^2 \\
\leq \frac{1}{4M^2}\|(V-P)\chi_{j,\beta}\|^2 \leq
\frac{1}{4}.
\end{eqnarray*}
Hence by Parseval's relation, we obtain
$$
\sum_{N \in A}|c(N,j,\beta)|^2>\frac{3}{4}.
$$
Using the fact that the number of indexes $N$ in $A$ is less than
$\rho^{d\alpha}$ and by the relation $N \notin B
\hspace{.2in}\Rightarrow \hspace{.2in}
|c(N,j,\beta)|<\rho^{-q\alpha}$, we have
$$
\sum_{N \in A\setminus
B}|c(N,j,\beta)|^2<\rho^{d\alpha}\rho^{-q\alpha}<\rho^{-\alpha},
$$
since $\alpha<\frac{1}{d+20}. $ On the other hand by the relation $A=(A\setminus B)\bigcup (A\bigcap B)$ and the above
inequalities, we get
$$
\frac{3}{4}<\sum_{N \in A}|c(N,j,\beta)|^2=\sum_{N \in A\setminus
B}|c(N,j,\beta)|^2+\sum_{N \in A\bigcap B}|c(N,j,\beta)|^2,
$$
which implies
\begin{equation}\label{AB}
\sum_{N \in A\bigcap
B}|c(N,j,\beta)|^2>\frac{3}{4}-\rho^{-\alpha}>\frac{1}{2}.
\end{equation}
Now, suppose that $A\bigcap B\bigcap C = \emptyset$, i.e., for all
$N \in A\bigcap B $, the condition (iii) does not hold. Then by
(\ref{AB}) and Bessel's inequality, we have
\begin{eqnarray*}
\frac{1}{2}<\sum_{N \in A\bigcap B}|c(N,j,\beta)|^2\leq \sum_{N
\in A\bigcap
B}\frac{1}{2p_2}\sum_{i=1}^{p_2}\left|\left \langle\psi_N,\frac{h_i}{\|h_i\|}\right\rangle \right|^2
\\
=\frac{1}{2p_2}\sum_{i=1}^m \sum_{N \in A\bigcap B}\left|
\left\langle\psi_N,\frac{h_i}{\|h_i\|}\right\rangle\right|^2 < \frac{1}{2p_2}\sum_{i=1}^{p_2}
\left\|\frac{h_i}{\|h_i\|}\right\|^2=\frac{1}{2}\;,
\end{eqnarray*}
which is a contradiction.
\item[(b)] The definition of $\lambda_{j,\beta}$ gives
\begin{align}\label{lambdajjk}
|\lambda_{j,\beta}-\lambda_{j^k,\beta^k}|&=
||\beta|^2+\lambda_j-|\beta+\beta_1+...+\beta_k|^2-\lambda_{j^k}|
\nonumber \\
&\geq
|||\beta|^2-|\beta+\beta_1+...+\beta_k|^2|-|\lambda_j-\lambda_{j^k}||.
\end{align}
The condition of the lemma $(\beta_1,j_1)\in Q(\rho^{\alpha},6r_1),(\beta_k,j_k)\in Q(\rho^{\alpha},6r_k)$ and the relation \newline
$\beta+j\delta\in V_{\delta}(\rho^{\alpha_1})\setminus E_2$
together with $|j\delta|<c_{11}\rho^{\alpha_1}$ (see \eqref{single}) and
$|j_i\delta|<c_{12}\rho^{\alpha_1}$ (see \eqref{single2})  imply that
\begin{align*}
\rho^{\alpha_2}&<||\beta|^2+|j\delta|^2-|\beta^k|^2-|j^k\delta|^2|  \\
&<||\beta|^2-|\beta^k|^2|+c_{9}\rho^{\alpha_1}, \hspace{.2in}
\beta_1+...+\beta_k\neq 0,
\end{align*}
since $\beta,\beta_1,...,\beta_k$ are orthogonal to $\delta$. That
is, we have
\begin{eqnarray*}
||\beta|^2-|\beta_k|^2|>c_{13}\rho^{\alpha_2}.
\end{eqnarray*}
This last inequality together with (\ref{lambdajjk}) and the
asymptotic formula \eqref{eigenvalue} give
$$
|\lambda_{j,\beta}-\lambda_{j^k,\beta^k}|>c_{14}\rho^{\alpha_2}.
$$
\end{itemize}
\end{proof}
\begin{center}
\large
Asymptotic Formulas
\end{center}
Now we consider the following function
\begin{align}\label{decom7}
h_i(x)=\sum_{\left(\beta_1,j_1\right)\in {\boldsymbol{Q}(\rho^{\alpha},6r_1)}\atop\left(\beta_2,j_2\right)\in {\boldsymbol{Q}(\rho^{\alpha},6r_2)}}\frac{A\left(j,\beta,j^{1},\beta^{1}\right)A\left(j^1,\beta^1,j^2,\beta^2\right)\chi_{j^{(2)},\beta^{(2)}}}{\left(\lambda_{j,\beta}-\lambda_{j^1,\beta ^1}\right)^i},\;\;\; 1\leq i \leq p_2.
\end{align}
Since $\left\{\chi_{j^{(2)},\beta^{(2)}}(x)\right\}$ is a total system and $\beta_1\neq 0$ by \eqref{absolute} and \eqref{cond6}, we have
\begin{align}\nonumber
\sum_{(j^{'},\beta^{'})}\lvert \langle h_i(x), \chi_{j^{'},\beta^{'}  }\rangle \lvert^2&=\sum_{\left(\beta_1,j_1\right)\in {\boldsymbol{F}(\rho^{\alpha},6r_1)}\atop\left(\beta_2,j_2\right)\in {\boldsymbol{F}(\rho^{\alpha},6r_2)}}\frac{\lvert A\left(j,\beta,j^{1},\beta^{1}\right)A\left(j^1,\beta^1,j^2,\beta^2\right)\lvert^2}{\lvert \left(\lambda_{j,\beta}-\lambda_{j^1,\beta ^1}\right)^i \lvert^2}\\ \label{decom9}&\leq c_{12}\; \rho^{-2i\alpha_2},
\end{align}
i.e., $h_i(x)\in L_2^m(F)$ and $\rVert h_i(x) \rVert=O(\rho^{-i\alpha_2}),\;\;\;\forall i=1,2,\ldots,p_2.$
\begin{theorem}
For every eigenvalue  $\lambda_{j,\beta}$ of the operator $L(P(s))$ with $\beta+j\delta\in V_{\delta}^{'}(\rho^{\alpha_1})$, there exists an eigenvalue $\Lambda_{N}$ of the operator $L(V)$ satisfying
\begin{align}\label{result1}
\Lambda_{N}=\lambda_{j,\beta}+O\left(\rho^{-\alpha_2}\right).
\end{align}
\end{theorem}
\begin{proof}
By Lemma 3, for the chosen $h_i(x), i=1,2,\ldots,p_2$ in \eqref{decom7}, there exists a number $N$, satisfying $(i),(ii),(iii).$ Since $(\beta_1,j_1)\in Q(\rho^{\alpha},6r_1)$, by part (b) of Lemma 3, we have
\begin{align*}
\lvert   \lambda_{j,\beta}-\lambda_{j^1,\beta^1}\lvert>c_{15}\rho^{\alpha_2}.
\end{align*}
The above inequality together with (i) imply
\begin{align*}
\lvert \Lambda_{N}-\lambda_{j^1,\beta^1}\lvert>c_{16}\rho^{\alpha_2}.
\end{align*}
Using the following well known decomposition
\begin{align*}
\frac{1}{\left[ \Lambda_{N}-\lambda_{j^1,\beta^1}\right]}=\sum_{i=1}^{p_2}\frac{\left[ \Lambda_{N}-\lambda_{j,\beta}\right]^{i-1}}{\left[ \lambda_{j,\beta}-\lambda_{j^1,\beta^1} \right]^i}+O\left(\rho^{-(p_2+1)\alpha_2}\right),
\end{align*}
and \eqref{decom7}, we see that the formula \eqref{decom8} can be written as
\begin{align*}
\left(\Lambda_{N}-\lambda_{j,\beta}\right)c(N,j,\beta)&=O(\rho^{-p\alpha})+\sum_{\left(\beta_1,j_1\right)\in {\boldsymbol{F}(\rho^{\alpha},6r_1)}\atop\left(\beta_2,j_2\right)\in {\boldsymbol{F}(\rho^{\alpha},6r_2)}}\frac{A\left(j,\beta,j^{1},\beta^{1}\right)A\left(j^1,\beta^1,j^2,\beta^2\right)\langle\psi_N,\chi_{j^2,\beta^2}\rangle}{\Lambda_{N}-\lambda_{j^1,\beta ^1}}\\
&=\sum_{i=1}^{p_2}\left[ ( \Lambda_{N}-\lambda_{j,\beta})^{i-1}\left\langle \psi_{N},\frac{h_i}{\lVert h_i \lVert}\right\rangle\right] \lVert h_i \lVert+O\left( \rho^{-(p_2+1)\alpha_2} \right).
\end{align*}
Now dividing both sides of the last equation by $c(n,j,\beta)$ and using $(ii), (iii)$, we have
\begin{align*}
&\lvert \Lambda_{N}-\lambda_{j,\beta}\lvert\leq O\left( \rho^{-(p_2+1)\alpha_2+q\alpha} \right)+ \\&\frac{\left\lvert
\left\langle\psi_{N},\frac{h_1}{\lVert h_1 \lVert}\right\rangle\right\lvert}{\lvert c(N,j,\beta) \lvert}\lVert h_1 \lVert + \frac{\lvert \Lambda_{N}-\lambda_{j,\beta} \lvert\; \left \lvert \left \langle\psi_{N},\frac{h_2}{\lVert h_2 \lVert}\right\rangle\right\lvert}{\lvert c(N,j,\beta) \lvert}\lVert h_2 \lVert +\ldots+\frac{\lvert \Lambda_{N}-\lambda_{j,\beta} \lvert^{(p_2-1)}\;\left\lvert \left\langle\psi_{N},\frac{h_{p_2}}{\lVert h_{p_2} \lVert} \right\rangle\right\lvert}{\lvert c(N,j,\beta) \lvert}\lVert h_{p_2} \lVert\\&\leq (2p_2)^\frac{1}{2}\left(\lVert h_1 \lVert+2M \lVert h_2 \lVert+\ldots+(2M)^{p_2-1}\lVert h_{p_2} \lVert\right)+O\left( \rho^{-(p_2+1)\alpha_2+q\alpha} \right).
\end{align*}
Hence by \eqref{decom9}, we obtain
\begin{align*}
\Lambda_{N}=\lambda_{j,\beta}+O\left(\rho^{-\alpha_2}\right),
\end{align*}
since $(p_2+1)\alpha_2-q \alpha>\alpha_2$. Theorem is proved.
\end{proof}
It follows from \eqref{cond6} and \eqref{result1} that the triples $(N,j^k,\beta^k)$ for $k=1,2,\ldots,p_1,$ satisfy the iterability condition \eqref{cond1}. By \eqref{decom4} instead of $j^{'},\beta^{'}$ and $r$ taking $j^2,\beta^2$ and $r_3$, we have
\begin{align}\label{c2}
c(N,j^2,\beta^2)=\sum_{(\beta_3,j_3)\in Q(\rho^{\alpha},6r_3)}\frac{A(j^2,\beta^2,j^3,\beta^3)(\psi_{N},\chi_{j^3,\beta^3})}{\Lambda_{N}-\lambda_{j^2,\beta^2}}+O(\rho^{-p\alpha}).
\end{align}
To obtain the other terms of the asymptotic formula of $\Lambda_{N}$, we iterate the formula \eqref{decom8}. Now we isolate the terms with multiplicand $c(N,j,\beta)$ in the right hand side of \eqref{decom8}.
\begin{align}\nonumber
(\Lambda_{N}-\lambda_{j,\beta})c(N,j,\beta)=O(\rho^{-p\alpha})&+\sum_{\substack{\left(\beta_1,j_1\right)\in {\boldsymbol{Q}(\rho^{\alpha},6r_1)}\\\left(\beta_2,j_2\right)\in {\boldsymbol{Q}(\rho^{\alpha},6r_2)}\\ (j+j_1+j_2,\beta+\beta_1+\beta_2)=(j,\beta)}}\frac{A\left(j,\beta,j^{1},\beta^{1}\right)A\left(j^1,\beta^1,j,\beta\right)}{\Lambda_{N}-\lambda_{j^1,\beta ^1}}c(N,j,\beta)\\ &\label{sumc2}+\sum_{\substack{\left(\beta_1,j_1\right)\in {\boldsymbol{Q}(\rho^{\alpha},6r_1)}\\\left(\beta_2,j_2\right)\in {\boldsymbol{Q}(\rho^{\alpha},6r_2)}\\ (j+j_1+j_2,\beta+\beta_1+\beta_2)\neq(j,\beta)}}\frac{A\left(j,\beta,j^{1},\beta^{1}\right)A\left(j^1,\beta^1,j^2,\beta^2\right)}{\Lambda_{N}-\lambda_{j^1,\beta ^1}}c(N,j^2,\beta^2).
\end{align}
Substituting the equation \eqref{c2} into the second sum of the equation \eqref{sumc2}, we get
\begin{align}\nonumber
(\Lambda_{N}-\lambda_{j,\beta})c(N,j,\beta)&=\sum_{\substack{\left(\beta_1,j_1\right)\in {\boldsymbol{Q}(\rho^{\alpha},6r_1)}\\\left(\beta_2,j_2\right)\in {\boldsymbol{Q}(\rho^{\alpha},6r_2)}\\ (j^2,\beta^2)=(j,\beta)}}\frac{A\left(j,\beta,j^{1},\beta^{1}\right)A\left(j^1,\beta^1,j,\beta\right)}{\Lambda_{N}-\lambda_{j^1,\beta ^1}}c(N,j,\beta)\\ \nonumber &\label{sumc3}+\sum_{\substack{\left(\beta_1,j_1\right)\in {\boldsymbol{Q}(\rho^{\alpha},6r_1)}\\\left(\beta_2,j_2\right)\in {\boldsymbol{Q}(\rho^{\alpha},6r_2)}\\ (j^2,\beta^2)\neq(j,\beta)\\     \left(\beta_3,j_3\right)\in {\boldsymbol{Q}(\rho^{\alpha},6r_3)} }}\frac{A\left(j,\beta,j^{1},\beta^{1}\right)A\left(j^1,\beta^1,j^2,\beta^2\right)A\left(j^2,\beta^2,j^3,\beta^3\right)}{(\Lambda_{N}-\lambda_{j^1,\beta ^1})(\Lambda_{N}-\lambda_{j^2,\beta ^2})}c(N,j^3,\beta^3)\\&+O(\rho^{-p\alpha}).
\end{align}
Again isolating terms $c(N,j,\beta)$ in the last sum of the equation \eqref{sumc3}, we obtain
\begin{align}\nonumber
(\Lambda_{N}-\lambda_{j,\beta})c(N,j,\beta)&=[ \sum_{\substack{\left(\beta_1,j_1\right)\in {\boldsymbol{Q}(\rho^{\alpha},6r_1)}\\\left(\beta_2,j_2\right)\in {\boldsymbol{Q}(\rho^{\alpha},6r_2)}\\ (j^2,\beta^2)=(j,\beta)}}\frac{A\left(j,\beta,j^{1},\beta^{1}\right)A\left(j^1,\beta^1,j,\beta\right)}{\Lambda_{N}-\lambda_{j^1,\beta ^1}}\\ &+ \sum_{\substack{\left(\beta_1,j_1\right)\in {\boldsymbol{Q}(\rho^{\alpha},6r_1)}\\\left(\beta_2,j_2\right)\in {\boldsymbol{Q}(\rho^{\alpha},6r_2)}\\\left(\beta_3,j_3\right)\in {\boldsymbol{Q}(\rho^{\alpha},6r_3)}\\ (j^2,\beta^2)\neq(j,\beta)\\ (j^3,\beta^3)=(j,\beta) }}\frac{A\left(j,\beta,j^{1},\beta^{1}\right)A\left(j^1,\beta^1,j^2,\beta^2\right)A\left(j^2,\beta^2,j,\beta\right)}{(\Lambda_{N}-\lambda_{j^1,\beta ^1})(\Lambda_{N}-\lambda_{j^2,\beta ^2})}]c(N,j,\beta)\nonumber\\
&+\sum_{\substack{\left(\beta_1,j_1\right)\in {\boldsymbol{Q}(\rho^{\alpha},6r_1)}\\\left(\beta_2,j_2\right)\in {\boldsymbol{Q}(\rho^{\alpha},6r_2)}\\\left(\beta_3,j_3\right)\in {\boldsymbol{Q}(\rho^{\alpha},6r_3)}\\ (j^2,\beta^2)\neq(j,\beta)\\ \nonumber (j^3,\beta^3)\neq(j,\beta) }}\frac{A\left(j,\beta,j^{1},\beta^{1}\right)A\left(j^1,\beta^1,j^2,\beta^2\right)A\left(j^2,\beta^2,j^3,\beta^3\right)}{(\Lambda_{N}-\lambda_{j^1,\beta ^1})(\Lambda_{N}-\lambda_{j^2,\beta ^2})}c(N,j^3,\beta^3)\\&+O(\rho^{-p\alpha}).
\end{align}
In this way, iterating $2p$ times, we get
\begin{align}\label{2p}
(\Lambda_{N}-\lambda_{j,\beta})c(N,j,\beta)=\left[\sum_{k=1}^{2p}S_k^{'}\right]c(N,j,\beta)+C^{'}_{2p}+O(\rho^{-p\alpha}),
\end{align}
where
\begin{align}\label{sk}
S_k^{'}(\Lambda_{N},\lambda_{j,\beta})=\sum_{\substack{\left(\beta_1,j_1\right)\in {\boldsymbol{Q}(\rho^{\alpha},6r_1)}\\\left(\beta_{k+1},j_{k+1}\right)\in {\boldsymbol{Q}(\rho^{\alpha},6r_{k+1})}\\ (j^{k+1},\beta^{k+1})=(j,\beta)\\ (j^s,\beta^s)\neq(j,\beta),\;s=2,\ldots,k }}\left(\prod_{i=1}^{k}\frac{ A\left(j^{i-1},\beta^{i-1},j^i,\beta^{i}\right)}{(\Lambda_{N}-\lambda_{j^i,\beta ^i})}\right)A\left(j^k,\beta^k,j,\beta\right)
\end{align}
and
\begin{align}\label{ck}
C_k^{'}=\sum_{\substack{\left(\beta_1,j_1\right)\in {\boldsymbol{Q}(\rho^{\alpha},6r_1)}\\\left(\beta_{k+1},j_{k+1}\right)\in {\boldsymbol{Q}(\rho^{\alpha},6r_{k+1})}\\  (j^s,\beta^s)\neq(j,\beta),\;s=2,\ldots,k+1 }}\left(\prod_{i=1}^{k}\frac{ A\left(j^{i-1},\beta^{i-1},j^i,\beta^{i}\right)}{(\Lambda_{N}-\lambda_{j^i,\beta ^i})}\right)A\left(j^k,\beta^k,j^{k+1},\beta^{k+1}\right)c(N,j^{k+1},\beta^{k+1}).
\end{align}
Now we estimate $S_k^{'}$ and $C_k^{'}$. For this, we consider the terms which appear in the denominators of \eqref{sk} and \eqref{ck}. By the conditions under the summations in \eqref{sk} and \eqref{ck}, we have
$j_1+j_2+\ldots+j_i\neq 0$ or $\beta_1+\beta_2+\ldots+\ldots\beta_i\neq 0,$ for $i=2,3,\ldots,k.$
\\
\\
If $\beta_1+\beta_2+\ldots+\ldots\beta_i\neq 0,$ then by \eqref{cond6} and \eqref{result1}, we have
\begin{align}\label{neq}
\lvert \Lambda_{N}-\lambda_{j^i,\beta ^i} \lvert > \frac{1}{2}\rho^{\alpha_2}.
\end{align}
If $\beta_1+\beta_2+\ldots+\ldots\beta_i= 0,$ i.e.,
$j_1+j_2+\ldots+j_i\neq 0$, then by a well-known theorem
\begin{align}\nonumber
\lvert \lambda_{j,\beta}-\lambda_{j^i,\beta ^i} \lvert =\lvert \mu_j-\mu_{j^i} \lvert > c_{17},
\end{align}
hence by \eqref{result1}, we obtain
\begin{align}\label{eq}
\lvert \Lambda_{N}-\lambda_{j^i,\beta ^i} \lvert > \frac{1}{2}c_{18}.
\end{align}
Since $\beta_k\neq 0$ for all $k \leq 2p$, the relation $\beta_1+\beta_2+\ldots+\ldots\beta_i= 0$ implies    $\beta_1+\beta_2+\ldots+\ldots\beta_{i\pm1 }\neq 0$. Therefore the number of multiplicands $\Lambda_{N}-\lambda_{j^i,\beta ^i} $ in  \eqref{ck} satisfying \eqref{neq} is no less then $p$. Thus, by \eqref{absolute}, \eqref{neq} and \eqref{eq}, we get
\begin{align}\label{order}
S_1^{'}=O(\rho^{-\alpha_2}),\;\;\;\; C_{2p}^{'}=O(\rho^{-p\alpha_2})
\end{align}
\begin{theorem}
\begin{itemize}
\item[(a)]
For every eigenvalue $\lambda_{j,\beta}$ of $L(P(s))$ such that $\beta+j\delta
\in V_{\delta
}^{'}(\rho^{\alpha_1})$, there exists an eigenvalue $\Lambda_{N}$ of the operator $L(V)$ satisfying
\begin{align}\label{a1}
\Lambda_{N}=\lambda_{j,\beta}+E_{k-1}+O(\rho^{-k\alpha_2}),
\end{align}
where $E_0=0$, $E_s=\sum\limits_{k=1}^{2p}S_{k}^{'}(E_{s-1}+\lambda_{j,\beta},\lambda_{j,\beta}),\;\;\;s=1,2,\ldots$
\item[(b)]
If \begin{align}\label{b1}
\lvert \Lambda_{N}-\lambda_{j,\beta}\lvert<c_{19}
\end{align}
and
\begin{align}\label{b2}
\lvert c(N,j,\beta) \lvert>\rho^{-q\alpha}
\end{align}
hold then $\Lambda_{N}$ satisfies \eqref{a1}.
\end{itemize}
\end{theorem}
\begin{proof}
By Lemma \eqref{lemma 3} $(a)-(b)$, there exists $N$ satisfying the conditions \eqref{b1} and \eqref{b2} in part $(b)$. Hence it sufficesto prove part $(b)$. By \eqref{cond6} and \eqref{b1}, the triples $(N,j^k,\beta^k)$ satisfy the iterability condition in \eqref{cond1}. Hence we can use \eqref{2p} and \eqref{order}. Now we prove the theorem by induction:
\\
\\
For $k=1$, to prove \eqref{a1}, we divide both sides of the equation \eqref{2p} by $c(N,j,\beta)$ and use the estimations \eqref{order}.
\\
\\
Suppose that \eqref{a1} holds for $k=s$, i.e.,
\begin{align}\label{s}
\Lambda_{N}=\lambda_{j,\beta}+E_{s-1}+O(\rho^{-s\alpha_2}).
\end{align}
To prove that \eqref{a1} is true for $k=s+1,$ in \eqref{2p} we substitute the expression \eqref{s} for $\Lambda_{N}$ into $\sum\limits_{k=1}^{2p}S_{k}^{'}(\Lambda_{N},\lambda_{j,\beta})$, then we get
\begin{align}
(\Lambda_{N}-\lambda_{j,\beta})c(N,j,\beta)=\left[\sum_{k=1}^{2p}S_k^{'}\left(\lambda_{j,\beta}+E_{s-1}+O(\rho^{-s\alpha_2}),\lambda_{j,\beta}\right)\right]c(N,j,\beta)+C^{'}_{2p}+O(\rho^{-p\alpha})
\end{align}
dividing the both sides of the last equality by $c(N,j,\beta)$ and using  Lemma \eqref{lemma 3}-$(ii)$, we obtain
\begin{align}\label{substii}
\Lambda_{N}=\lambda_{j,\beta}+\sum_{k=1}^{2p}S_k^{'}\left(\lambda_{j,\beta}+E_{s-1}+O(\rho^{-s\alpha_2}),\lambda_{j,\beta}\right)+O(\rho^{-(p-q)\alpha}).
\end{align}
Now we add and subtract the term $\sum\limits_{k=1}^{2p}S_k^{'}\left(E_{s-1}+\lambda_{j,\beta},\lambda_{j,\beta}\right)$ in \eqref{substii}, then we have
\begin{align}\label{subtlast}
\Lambda_{N}=\lambda_{j,\beta}+E_s+O(\rho^{-(p-q)\alpha})+\left[\sum_{k=1}^{2p}S_k^{'}\left(\lambda_{j,\beta}+E_{s-1}+O(\rho^{-s\alpha_2}),\lambda_{j,\beta}\right)-\sum_{k=1}^{2p}S_k^{'}\left(E_{s-1}+\lambda_{j,\beta},\lambda_{j,\beta}\right)\right].
\end{align}
Now, we first prove that $E_j=O(\rho^{-\alpha_2})$ by induction. $E_0=0$. Suppose that $E_{j-1}=O(\rho^{-\alpha_2})$, then $a=\lambda_{j,\beta}+E_{j-1}$ satisfies \eqref{neq} and \eqref{eq}. Hence we get
\begin{align}\label{s1}
S_1^{'}(a,\lambda_{j,\beta})=O(\rho^{-\alpha_2})\Rightarrow E_j=O(\rho^{-\alpha_2}).
\end{align}
To prove the theorem, we need to show that the expression  in the square brackets in \eqref{subtlast} is equal to $O(\rho^{-(s+1)\alpha_2})$. This can be easily checked by \eqref{s1} and the obvious relation
\begin{align}
\frac{1}{\lambda_{j,\beta}+E_{s-1}+O(\rho^{-s\alpha_2})-\lambda_{j^k,\beta^k}}-\frac{1}{\lambda_{j,\beta}+E_{s-1}+\lambda_{j^k,\beta^k}}=O(\rho^{-(s+1)\alpha_2}),
\end{align}
for $\beta^k\neq\beta$. The theorem is proved.
\end{proof}

\end{document}